\documentclass[reqno]{amsart}

\usepackage{amsmath}
\usepackage{amsthm}
\usepackage{amsfonts}
\usepackage{amssymb}
\usepackage{enumerate}
\usepackage{color}
\usepackage{graphicx}

\numberwithin{equation}{section}

\newcommand{\E}{\mathbb{E}}

\newcommand{\C}{\mathbb{C}}
\newcommand{\R}{\mathbb{R}}

\newcommand{\Tr}{\mathrm{Tr}}
\newcommand{\tr}{\mathrm{tr}}

\renewcommand{\d}{\partial}

\renewcommand\Im{\operatorname{Im}}

\renewcommand\P{\mathbb{P}}

\theoremstyle{plain}
  \newtheorem{theorem}{Theorem}[section]

  \newtheorem{proposition}[theorem]{Proposition}
  
  \newtheorem{lemma}[theorem]{Lemma}
  \newtheorem{corollary}[theorem]{Corollary}

\theoremstyle{definition}

  \newtheorem{remark}[theorem]{Remark}

\begin{document}

\title{Eigenvalues of block structured asymmetric random matrices}

\author[J. Aljadeff]{Johnatan Aljadeff}
\address{Department of Physics, UCSD; Computational Neurobiology Laboratory, The Salk Institute for Biological Studies}
\email{aljadeff@ucsd.edu}

\author[D. Renfrew]{David Renfrew}
\thanks{D. Renfrew is partly supported by NSF grant DMS-0838680}
\address{Department of Mathematics, UCLA  }
\email{dtrenfrew@math.ucla.edu}

\author[M. Stern]{Merav Stern}
\address{The Edmond and Lily Safra Center for Brain Sciences, Hebrew University; Department of Neuroscience, Columbia University}
\email{merav.stern@mail.huji.ac.il}

\begin{abstract}
We study the spectrum of an asymmetric random matrix with block structured variances. The rows and columns of the random square matrix are divided into $D$ partitions with arbitrary size (linear in $N$). The parameters of the model are the variances of elements in each block, summarized in $g\in\mathbb{R}^{D\times D}_+$. Using the Hermitization approach and by studying the matrix-valued Stieltjes transform we show that these matrices have a circularly symmetric spectrum, we give an explicit formula for their spectral radius and a set of implicit equations for the full density function. We discuss applications of this model to neural networks.
\end{abstract}

\maketitle

\section{Introduction}

Non-Hermitian random matrices have proven to be a useful theoretical tool to study, for example, physical and biological systems \cite{Sompolinsky1988,Stephanov1996}. The spectrum of different matrix models can be used to make statements and predictions for the behavior of those systems. There is often a constant tension between having to accurately represent the system of interest and wanting the matrix model to be tractable. In recent years this tension served as motivation to study the spectrum of generalized random matrix models that include structure \cite{Rajan2006,Wei2012,Ahmadian2013}.

Along these lines, recently we defined a square non-Hermitian random model by partitioning the matrix into a finite number of blocks and letting the variance of elements in each block be an independent parameter, see section \ref{model} for a formal definition. We will refer to these here as the heterogeneous model, in contrast to the homogeneous model where all the elements are drawn from the same distribution (giving Girko's circular law). The heterogeneous random matrix can be thought to represent the connectivity in a neural network with multiple cell-types, where the connectivity statistics are cell-type dependent \cite{Aljadeff2014}.

Previously, by using a mean-field approach to study the heterogeneous network, we derived a formula for an effective parameter axis along which a critical point is located. As in the homogeneous network \cite{Sompolinsky1988}, at this critical point the network undergoes a phase transition and its dynamics transform from having a single stable fixed point to chaos. This parameter was identified as the spectral radius of the heterogeneous matrix. Thus, using non-rigorous methods common in the physics literature, we obtained a formula for the spectral radius. This approach did not provide any information about the density.

Surprisingly, unlike in a matrix model where the variance depends only on the row or column index \cite{Rajan2006}, the spectral radius of the block-structured model can be smaller, equal to, or larger than the expectation value of its Hilbert-Schmidt norm. In the context of the application to neural networks this means that a particular block-structured organization of the connections in the network can dramatically affect the global dynamics and the computational capacity of the network \cite{Aljadeff2014}.

Random matrices with block structure have been studied in several different contexts. In the Hermitian case, block random  matrices have been used to study MIMO communication channels, see for example \cite{FOBS}. In the non-Hermitian case, block random matrices with each block having a possibly different non-zero mean, but each entry of matrix having the same variance, was considered in \cite{F}. In \cite{NO}, Girko's circular law was proved for block random matrices with dependencies between blocks but each entry having the same variance.

Here, using the Hermitization approach and by studying the matrix-valued Stieltjes transform, we prove that the support of the spectral density is given by the formula in \cite{Aljadeff2014}, and find a set of implicit equations for the density that can be solved numerically. For certain low dimensional parameterizations of the model it is possible to use these equations to approximate the parameter dependence of interesting quantities related to the spectrum, such as the mass on a given annulus.

It is possible to compute the spectrum of the matrix studied here when a finite rank perturbation is added. Knowledge of the spectrum is often necessary but not sufficient to understand the physical model's behavior. In the example of neural networks with cell-type-dependent connectivity statistics that motivated this work, adding a finite rank perturbation allows treatment of matrices where all the elements in each column have the same sign. These are useful in modeling networks that obey Dale's principle - stating that real neurons are either excitatory or inhibitory. However, the mean-field characterization of the dynamics in this model remains a subject for future research.

\section{Main result}
\label{model}

We now introduce our model and main result. Let $J_N^0$ be an $N \times N$ matrix with iid random entries with zero mean, variance $1/N$, and finite fourth moment. Let $g$ be a $D \times D$ matrix with real, positive entries. Let $\alpha$ be a $D$ dimensional vector such that
\[ \alpha_i > 0, \quad\sum_{i=1}^D \alpha_i =1\]
and let \[c_i = \left\{ c \left| \frac{i}{N} \in \left( \sum_{d=1}^{c-1} \alpha_d , \sum_{d=1}^{c} \alpha_d \right]\right. \right\}  . \]
Let $X_N$ be an $N\times N$ random matrix whose $i,j$ entry is
\[ X_{ij} := g_{c_i c_j} J_{ij}^0 .\]
We denote by $X^{cd}$, for $1\leq c,d \leq D$, the $\lfloor \alpha_c N \rfloor \times \lfloor \alpha_d N \rfloor$ matrix, which is the $cd^{th}$ block of $X_N$. Generally, superscripts will refer to $O(D)$ quantities and subscripts to $O(N)$ quantities.

Our main result concerns the eigenvalues of $X_N$ in the limit $N \to \infty$ with $D$ fixed. We also fix $\alpha$ and $g$. Let $\lambda_1, \ldots, \lambda_N$ be the eigenvalues of $X_N$ and let $\mu_N = \frac{1}{N} \sum_{i=1}^{N} \delta_{\lambda_i}$ be the empirical spectral measure of $X_N$.

Before stating our main result, we recall that $\d_z = \frac{1}{2} (\d_x -\sqrt{-1} \d_y)$. Let $G$ be the $D \times D$ matrix with entries $G_{cd} = \alpha_c g^2_{cd}$ and $\widehat G_{cd} = \alpha_c g^2_{dc}$. The matrices $G$ and $\widehat G$ have positive entries, thus by the Perron-Frobenius Theorem the largest eigenvalue in magnitude of each matrix is a positive real number. Furthermore, the eigenvector associated with this eigenvalue, called the Perron-Frobenius eigenvector, has strictly positive entries.

We denote by $\rho(G)$ the spectral radius of $G$. In fact, this is also the spectral radius of $\widehat{G}$. This can be seen by letting $\alpha$ be the diagonal matrix whose $c^{th}$ entry is $\alpha_c$ and $g^{(2)}$ be the matrix whose $cd^{th}$ entry is $g^2_{cd}$ then we have
\[ \rho(G) = \rho(\alpha g^{(2)}) = \rho( g^{(2)*} \alpha^*) = \rho(\alpha g^{(2)*}) = \rho(\widehat G) .\]

The limiting density of $\mu_N$, is characterized by the $D$ dimensional vectors, $a$, $\widehat a$, and $b$, which will be functions of $z$, a point in the complex plane, and $\eta$ a regularization parameter.

\begin{theorem}
The empirical spectral measure of $X_N$ converges almost surely to a deterministic measure $\mu$.
The density of $\mu$ is radially symmetric and its support has radius $\sqrt{\rho(G)}$. The density of $\mu$ at $|z| \leq \sqrt{\rho(G)}$ is
$ \lim_{\eta = \sqrt{-1} t \to 0} \frac{-1}{\pi}\sum_{c=1}^{D} \alpha_c \partial_z b_c(z,\eta)$
where $a$, $\widehat a$, and $b$ are the unique positive solutions to: \[b_c(z,\eta) = \frac{ z a_c(z,\eta)}{[Ga(z,\eta)]_c},\]
\[ a_{c}(z,\eta) = \frac{ [G  a(z,\eta) ]_c+ \eta}{|z|^2 - ( [\widehat G \widehat a(z,\eta)]_c + \eta)([G a(z,\eta) ]_c + \eta) },\]
\[ \widehat a_c(z,\eta)= \frac{ [\widehat G \widehat a(z,\eta)]_c + \eta}{ |z|^2 -  ([ \widehat G \widehat a(z,\eta)]_c + \eta)( [Ga(z,\eta)]_c + \eta)  } .\]
\end{theorem}

The proof is divided into several steps. We characterize the limiting deterministic measure, $\mu$, in section \ref{MVST}. We use Lemma \ref{BClemma} to prove the convergence of $\mu_N \to \mu$ in sections \ref{deriv} and \ref{logi}. We show that the density of $\mu$ is radially symmetric with given spectral radius in section \ref{analysis}.

\begin{remark}
The assumption that the entries have a finite fourth moment is natural, under less assumptions one does not expect all of the eigenvalues of $X_N$ to stay within the support of $\mu$. Furthermore, repeating the lengthy computations of \cite[Section 6]{OR} shows that almost surely there are no eigenvalues of $X_N$ outside the support of $\mu$. Finally, the appendix of \cite{TV} discusses a weak invariance principle for the eigenvalues of non-iid random matrices. This result combined with the least singular value bound in Section \ref{logi} can be used to prove convergence in probability to the same limiting measure when the entries of the random matrix only have finite variance.
\end{remark}

\begin{remark}
Under certain restrictions, the matrix model we study here coincides with previously studied models. Rajan and Abbott \cite{Rajan2006} give an explicit formula for the spectrum in the case with $D=2$ and $g$ that depends on only one index (i.e. $g$ is equal to two copies of the same vector). Interestingly, one can show that when the entries of $g$ depend only on their column or row index, for any $D$, the spectral radius of $X_N$ is equal to the expectation value of its Hilbert-Schmidt norm: $\mathbb{E}\|X_N\|_{HS} = (\sum_{c,d=1}^D\alpha_c\alpha_d g_{cd}^2)^{\frac{1}{2}}$. This is in general not the case for our matrix model. When ${\rm rank}\{g\}=1$, the matrices studied by Wei and Ahmadian et al. \cite{Wei2012,Ahmadian2013} can be defined (under certain conditions) to coincide with our model.
\end{remark}

\begin{figure}[h]
\includegraphics[width=1\textwidth]{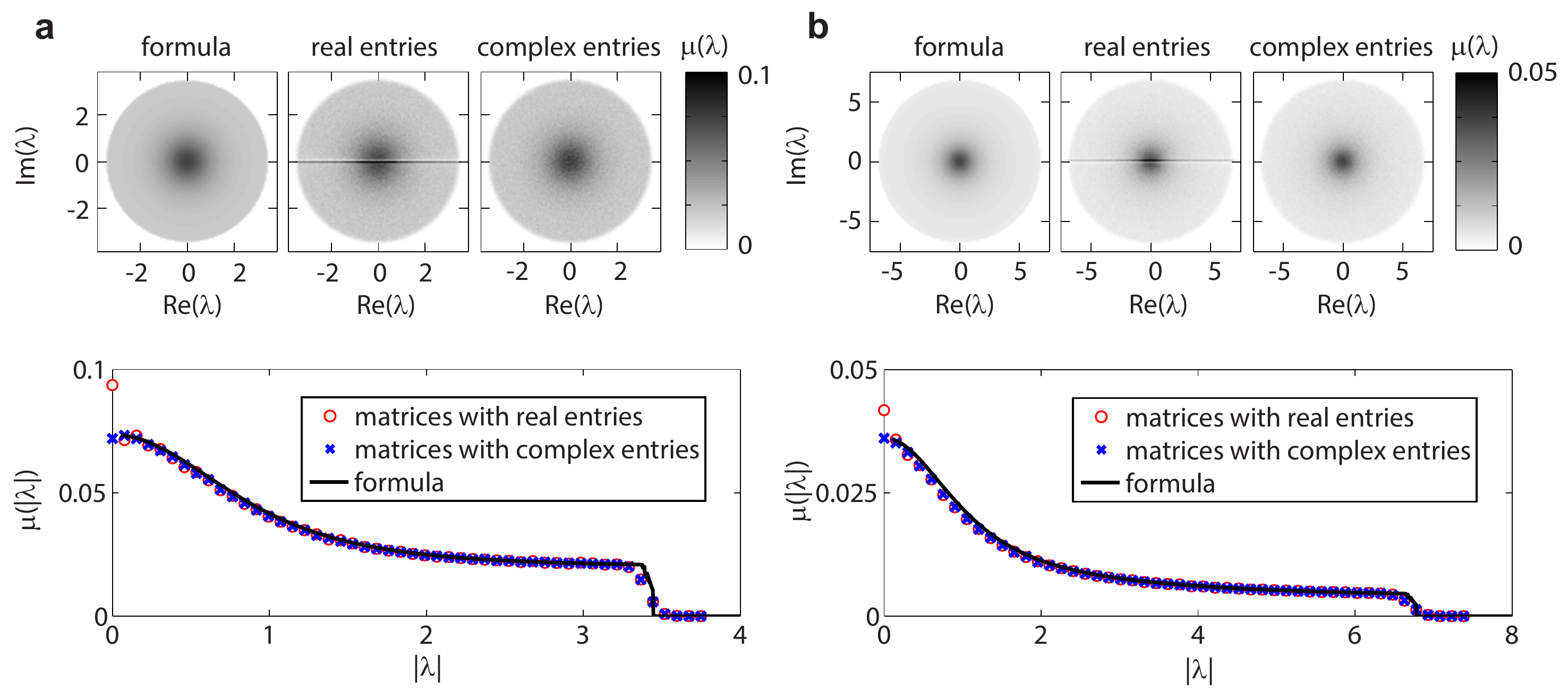}
\caption[]{Plot of example spectra in the complex plane (top) and of the radial part of the density (bottom). For each example we compared the numerical calculation from our formulae to direct computation of 1000 matrices with $N=1000$ and entries in $\mathbb{R}$ and $\mathbb{C}$. Note that for finite $N$ and matrices with real entries there is a large concentration of eigenvalues on the real axis. (a) An example with $D=2$, $\alpha=\left(0.3,0.7\right)$ and $g=\left(\begin{smallmatrix}1&2\\3&4\end{smallmatrix}\right)$. (b) An example with $D=3$, $\alpha=\left(0.25,0.30,0.45\right)$ and $g=\left(\begin{smallmatrix}1&2&3\\4&5&6\\7&8&9\end{smallmatrix}\right)$.\label{fig}}
\end{figure}

\section{Hermitization}

An empirical spectral measure, $\mu$, of a Hermitian matrix is often studied by considering its Stieltjes transform, defined by $  \int \frac{ d \mu(x)}{x - z }$ for $z$ with positive imaginary part. Since the eigenvalues of non-normal matrices can be complex, it is difficult to directly study the Stieltjes transform of the empirical spectral measure. Spectral instability of non-normal matrices introduces more difficulties in attempting to use standard hermitian techniques, see for example \cite[Section 11.1]{BS}. To circumvent these issues we follow the method of Hermitization pioneered by Girko, \cite{G}, and since refined by many authors, see for example \cite{B, BC, TV} and references within.

The Hermitization of $X_N$ is defined to be the $2N \times 2N$ matrix
\[ H_N := \begin{pmatrix} 0 & X_N \\ X_N^* & 0 \end{pmatrix} .\]
We also define the Hermitization of $X_N-z$ as
\[H_N(z) :=  \begin{pmatrix} 0 & X_N -z I_N \\ (X_N-z I_N)^* & 0 \end{pmatrix} \]
from which we define the resolvent
\begin{equation}
\label{defres}
R_N(q) := \begin{pmatrix}- \eta I_N & X_N -z I_N \\ (X_N-z I_N)^* &- \eta I_N \end{pmatrix}^{-1} = (H_N - q \otimes I_N)^{-1}
\end{equation}
where $q = \begin{pmatrix} \eta & z  \\ \overline z  & \eta \end{pmatrix}$.

Viewing $R_N(q)$ as an $2 \times 2$ block matrix with $N \times N$ sized matrix entries and taking the trace over each block leads to the $2 \times 2$ matrix-valued Stieltjes transform:
\[ \Gamma_N(q):= \begin{pmatrix}  a_N(q) & b_N(q) \\ c_N(q) & a_N(q) \end{pmatrix} :=(I_2 \otimes \tr_N) R_N(q)   \]
where $\tr_N:= \frac{1}{N} \Tr$ is the normalized trace of an $N \times N$ matrix. The 2 by 2 matrix $\Gamma_N(q)$ was introduced in the math literature in \cite{BC} and used in the physics literature previously.

We will often be interested in $\eta = \sqrt{-1} t$ for $t >0$, in which case $c_N = \overline{b_N}$. When we wish to emphasize the dependence on $z$ or $\eta$ we will replace $(q)$ with $(z,\eta)$ in the argument of $\Gamma_N$, $a_N$, $b_N$ or $c_N$.

Let $\nu_{z,N}$ the empirical spectral measure of $H_N(z)$. From $\nu_{z,N}$ the singular values of $X_N-z I_N$ can be recovered by noting that if $\sigma$ is a singular value of $X_N-z$ then $\pm \sigma$ is an eigenvalue of $H_N(z)$.
Since $R_N(z,\eta)$ is the resolvent of $H_N(z)$ at $\eta$, $a_N(z,\eta)$ is the Stieltjes transform of $\nu_{z,N}(x)$ at $\eta$.

Note that by direct calculation
\[ a_N(\eta,z) = \eta \tr_N(\eta^2 - (X_N-z) (X_N-z)^*)^{-1 }) .\]

In particular $a_N$ is purely imaginary and with positive imaginary part when $\eta = \sqrt{-1} t$ with $t>0$.

By the Stieltjes inversion formula, $\nu_{z,N}$ can be recovered from $a_N(q)$. Furthermore, $\mu_N$ can be recovered from $\nu_{z,N}$ by taking the Laplacian of the logarithmic potential, $U_N(z) :=  -\int_{\C} \log( s- z) d\mu_N(s)$, and using the following identities.
\begin{align*}
2 \pi \mu_N(z) &= \Delta \int_{\C} \log( s- z) d\mu_N(s) =  \Delta \frac{1}{N} \sum_{i=1}^N \log( \lambda_i -z) \\
&=  \Delta \frac{1}{N} \log( \det( X_N -z)) =  \Delta \int_{0}^{\infty} \log( x) d\nu_{z,N}(x)
\end{align*}

While $a_N(q)$ characterizes $\mu_N$, it can be difficult to compute its density from $\mu_N$. Fortunately, the density can also be recovered from $(X_N-z)^{-1}$, by noting that from Jacobi's formula $\d_{\overline z} U_N(z) = \frac{1}{2} \tr_N( (X_N-z)^*)^{-1} )$, and then applying $\d_z$ to $\tr_N( (X_N-z)^*)^{-1} )$. As noted above, the resolvent is not bounded uniformly, fortunately $\Gamma_N$ gives a regularization of the resolvent.

Once again by direct computation
\[ b_N(\eta,z) = -\tr_N( (X_N-z)(\eta^2 - (X_N-z) (X_N-z)^*)^{-1} ). \]
In the limit $\eta \to 0$ we recover the adjoint of the resolvent
\[ \lim_{\eta \to 0} b_N(\eta,z) =\tr_N( (X_N-z)^*)^{-1} ) .\]
The justification of this regularization in computing the limiting distribution is the content of the next lemma.

\begin{lemma}[Lemma 4.20, \cite{BC}]
\label{BClemma}
Let $(X_N)_{N\geq 1}$ be a sequence of random matrices. Assume for all $q \in \mathbb{H}_+$, there exists
\[ \Gamma(q) = \begin{pmatrix} a(q) & b(q) \\ c(q)  & a(q)   \end{pmatrix} \in  \mathbb{H}_+\]
such that for a.a. $z \in \C$, $\eta \in \C_+$, with $q = q(z,\eta)$,
\begin{enumerate}[(i)]
\item a.s. (respectively in probability) $\Gamma_{N}(q)$ converges to $\Gamma(q)$ as $N \to \infty$
\item a.s. (respectively in probability) $\log$ is uniformly integrable for $(\nu_{z,N})_{N\geq 1}$
\end{enumerate}
Then there exists a probability measure $\mu $ such that\\
\phantom{.}~ (j) a.s. (respectively in probability) $\mu_{N} \to \mu$ as $N \to \infty$\\
\phantom{.}~ (jj) a.s. (respectively in probability) $\mu = \frac{-1}{\pi} \lim_{q(z,\sqrt{-1} t): t \to 0} \d_z b(q)$.
\end{lemma}

For $\log$ to be uniformly integrable we require that 

\begin{equation}
\label{logieq}
\lim_{t \to \infty} \sup_{N\geq 1}  \int_{|\log(x)|>t} | \log(x)| d \nu_{z,N}(x)  =0.
\end{equation}

The remainder of this section and section \ref{deriv} are devoted defining $\Gamma(q)$ and showing that $\Gamma_N(q) \to \Gamma(q)$ almost surely. In section \ref{logi}, we show that $\log$ is almost surely uniformly integrable for $(\nu_{z,N})_{N\geq 1}$.

\subsection{Matrix-Valued Stieltjes transform}
\label{MVST}
Instead of directly working with the $2 \times 2$ matrix, $\Gamma_N(q)$, we use the block structure of $X_N$ and consider a $2D \times 2D$ matrix-valued Stieltjes transform, formed by taking partial traces over each block of the resolvent.
The resolvent $R_N$ is divided into $(2D)^2$ blocks using the partitions coming from $X_N$ in $\begin{pmatrix} 0 & X_N \\ X_N^* & 0 \end{pmatrix}$. Namely, the $cd^{th}$ block is of size $\lfloor \alpha_{c~ mod D} N \rfloor \times \lfloor \alpha_{d~ mod D} N \rfloor$.
We label the $cd^{th}$ block of $R^{cd}$, and its entries $R^{cd}_{ij}$ with $i,j$ running over the size of the block.

When possible we will omit the dependence of these matrices on $N$ and $q$.
Let $M_N(q)$ be a $2D\times 2D$ matrix with $c,d$ entry
\begin{equation}
\label{defMN}
M_N^{cd}(q) := \frac{1}{\alpha_c} \tr_N(R_N^{cd}(q)).
\end{equation}
From this matrix-valued Stieltjes transform the actual Stieltjes transform can be recovered by taking a weighted trace, $a_N(q) = \sum_{c=1}^D \alpha_c M_N^{cc}(q)$.

To describe the limiting Stieltjes transform, we use the theory of operator-valued semicircular elements, also used in, for example \cite{A,HFS,HT}. In section \ref{deriv}, we show that $M_N(q)$ approximately satisfies a self consistent equation, from which we can estimate the difference between $M_N(q)$ and the actual solution to this self consistent equation. The solution to the fixed point equation is the operator-valued Stieltjes transform of an operator-valued semicircular element. This semicircular element is determined by a completely positive map $\Sigma$. The operator-valued Stieltjes transform of this element is the solution to the equation
\begin{equation}
\label{Meq}
M(\tilde q) = - ( \tilde q+ \Sigma(M(\tilde q) ) )^{-1}
\end{equation}
with positive imaginary part for $\tilde q$ with positive imaginary part. Recall the  imaginary part of a matrix, $X$ is $\frac{1}{2 \sqrt{-1} } ( X - X^*)$. In \cite{HFS}, it is shown that there is a unique solution to \eqref{Meq} with this positivity condition. Furthermore, this solution can be found by iterating the map $W \to - (\tilde  q+ \Sigma(W)  )^{-1}$ and that the norm of the solution is bounded by $\| \Im(\tilde q)^{-1}\|$. When $\tilde q$ is taken to be $q \otimes I_D$, then $\Im(\tilde q) = \Im(\eta) I_{2 D}$. The uniqueness of solution and the bound on its norm will be used in the proof of Lemma \ref{expest}.

In our case, $M(\tilde q)$ is a $2D\times 2D$ matrix and $\Sigma$ is a linear operator on $2D \times 2D$ matrices, defined by

\[\Sigma(A)_{cd} = \delta_{cd} \sum_{e=1}^{D} \alpha_{c}  g_{ce}^2  A_{e+D,e+D}   \]
for $1\leq c,d \leq D$ and
\[\Sigma(A)_{cd} = \delta_{cd} \sum_{e=1}^{D} \alpha_{c}  g_{ec}^2  A_{e,e}  \]
for $D+1\leq c,d \leq 2D$ .
In the derivation of this equation we will see that $\Sigma(A)$ is a diagonal matrix as a result of the independence between matrix entries of $X_N$.

By iterating the map $W \to - ( q \otimes I_D + \Sigma(W)  )^{-1}$, with initial point $q \otimes I_D$, we see the solution will be of the form
\begin{equation}
\label{defM}
 M(q) =  \begin{pmatrix}  a_1 &  &   &b_1 &  &  \\
 & \ddots &  & & \ddots &  \\
 &  &  a_D & & & b_D \\
c_1 &  &   & \widehat a_{1} &  &  \\
 & \ddots &  & & \ddots &  \\
 &  &  c_D & & & \widehat a_{D}  \end{pmatrix}
 \end{equation}
 with the empty entries equal to zero.
Furthermore when $\eta=\sqrt{-1} t$, the diagonal entries will be purely imaginary, with positive imaginary part.

Inverting the matrix on the right side of \eqref{Meq} with our particular $\Sigma$ gives a system of equations for the entries of $M$ (written below). The first $D$ diagonal entries are acted on by the previously introduced matrix $G$, with entries $G_{cd} = \alpha_c g^2_{cd}$, and the remaining diagonal entries are acted on by $\widehat G$ which has entries $ \widehat G_{cd} = \alpha_c g^2_{dc}$.

\begin{equation}
\label{aeq}
 a_{c} = \frac{ [G  a ]_c+ \eta}{|z|^2 - ( [\widehat G \widehat a]_c + \eta) ([G a ]_c + \eta) }
 \end{equation}
\begin{equation}
\label{waeq}
 \widehat a_c= \frac{ [\widehat G \widehat a]_c + \eta}{ |z|^2 -  ( [\widehat G \widehat a]_c + \eta)( [Ga]_c + \eta)  }
 \end{equation}
 and
\begin{equation}
\label{beq} b_c = \frac{- z }{ |z|^2 - ( [\widehat G \widehat a]_c + \eta) ( [Ga]_c + \eta)  },
 \end{equation}
 for $1\leq c\leq D$.

Using Lemma \ref{BClemma} the limiting measure $\mu = \frac{-1}{\pi} \lim_{t \to 0}  \partial_z \sum_{c=1}^D \alpha_c b_c(z,\sqrt{-1} t) $.
If $\|Ga\| \to 0$ as $\eta \to 0$ then $b \to -1/\overline{z}$ whose derivative with respect to $z$ vanishes, otherwise $b_c \to  \frac{ -z a_c}{[Ga]_c}$.

\begin{remark}
In Fig. \ref{fig} we compare the full spectral density and the radial part computed numerically from the implicit equations (\ref{aeq}-\ref{waeq}) to direct diagonalization of instantiations with $D=2$ and $D=3$. The density is computed from the above equations by choosing a grid of $z$ values and, for each point, iterating the maps for $a, \widehat{a}$ until convergence. The solutions to these equations are substituted into \eqref{beq}, and finally the density is obtained by computing the gradient weighted by the appropriate model parameters.
\end{remark}

\section{Derivation of fixed point equation}
\label{deriv}

In this section we show $M_N(q)$ approximately solves equation \eqref{Meq} with the given $\Sigma$, from which the difference $\|M_N(q) - M(q)\|$ will be estimated.

We first introduce some notation.
Let $R_N^{(k)}(q)$ be the resolvent of $H_N(z)$ with the $k^{th}$ row and column of each block set to zero, let $R_N^{cd(k)}(q)$ be the $c,d$ block of this matrix,
and let $M_N^{cd(k)}(q) := \frac{1}{\alpha_c} \tr_N(R_N^{cd(k)}(q) )$.
Let $R_{N;11}(q)$ be the $2D \times 2D$ matrix with entries that are the $(1,1)$ entry of each block.
\[ R_{N;11} := \begin{pmatrix} R^{11}_{11} & R^{12}_{11} & \ldots & R^{2D,2D}_{11} \\ R^{21}_{11} & \ddots & & \vdots \\
\vdots &  &\ddots & \vdots \\ R_{11}^{2D,1} & \ldots &   \ldots & R_{11}^{2D,2D}       \end{pmatrix}  \]

We recall Schur's complement for computing entries of the inverse of a matrix.
Given a block matrix
\[ X = \begin{pmatrix} A & B \\ E & D \end{pmatrix} \]
then the entries of the upper left block of the inverse can be computed via the formula:
\[ \left(X^{-1} \right)_{11} = ( A - E D^{-1} B )^{-1}_{11} \]
or more generally, for a set $I$ such that $\{1 ,\ldots, N\} = I \cup I^c$
\[  \left(X^{-1} \right)_{I,I} = \left(X_{I,I} - X_{I,I^c} (X_{I^c,I^c})^{-1} X_{I^c,I} \right)^{-1} \]
We apply Schur's complement to $H_N - q \otimes I_N$ using the set $I$ formed by taking the index of the first entry of each block, leading to an expression for $R_{N;11}$.

The main estimate of this section is the contents of the next proposition. By Vitali's convergence theorem, it suffices to prove convergence of $\Gamma_N$ for $\eta = \sqrt{-1}t$, with $0 < t < T$, for some large $T$. By applying the Bai-Yin Theorem for the operator norm to each block, $\|X_N\| < T$ almost surely for $T$ sufficiently large (see for instance \cite[Chapter 5]{BS} or \cite[Section 2.3]{T}), so it suffices to prove converge for $z$ in a compact set.

\begin{proposition}
\label{mainest}
Let $T>0$ be sufficiently large and $|z| < T$, $\eta = \sqrt{-1} t $ with $0<|t| < T$. Then almost surely
$\| M_N(q) - M(q) \| =O(|\Im(\eta)|^{-5} N^{-1/4})$.
\end{proposition}

This estimate is not optimal, but will suffice for our purposes. It would be interesting to see if optimal bounds can be obtained as in \cite{BYY}.

Following a similar argument to \cite{ORSV}, Proposition \ref{mainest} will follow from Lemma \ref{expest} and Lemma \ref{concen}. We begin by deriving an equation for $\E[M_N(q)]$. Bounding the resulting error term gives a bound on $\| \E[M_N(q)] - M(q)\|$. We finish the section by bounding $\| \E[M_N(q)] - M_N(q)\|$.

It suffices to consider $R_{N;11}$ because
\[ \E[M_N] = \E[R_{N;11}] \]
by exchangeability.

Then, by Schur's complement
\begin{align}
\label{Rschur}
 &R_{N;11} =
- \left(-H_{11} + q\otimes I_D +
 H_{1\cdot}^{(1)} R_N^{(1)}  H_{\cdot 1}^{(1)} \right)^{-1}
 \end{align}

 Where $H_{11}$ is the $2D \times 2D$ with scalar entries
 \[ H_{11} = \begin{pmatrix}
0 &
\begin{matrix}
 X_{11}^{1,1} & \ldots & X_{11}^{1,D}  \\ \vdots  &   & \vdots\\  X_{11}^{D,1}   &\ldots  &  X_{11}^{D, D} \end{matrix}  \\
\begin{matrix}
X_{11}^{1,1} & \ldots & X_{11}^{D,1}  \\ \vdots  &   & \vdots\\  X_{11}^{1,D}   &\ldots  &  X_{11}^{D, D} \end{matrix}   &
0
\end{pmatrix}\]
and $H_{1\cdot}^{(1)}$, $H_{\cdot 1}^{(1)}$ are $2D \times 2D$ with vector entries

\[ H_{1\cdot}^{(1)} = \begin{pmatrix}
0 &
\begin{matrix}
 X_{1\cdot}^{1,1(1)} & \ldots & X_{1\cdot}^{1, D(1)}  \\ \vdots  &   & \vdots\\  X_{1\cdot}^{D,1 (1)} &\ldots  &  X_{1\cdot}^{D, D(1)} \end{matrix}  \\
\begin{matrix}
 X_{\cdot 1}^{1,1(1)} & \ldots &X_{\cdot 1}^{D, 1(1)}  \\ \vdots  &   & \vdots\\  X_{\cdot 1}^{1, D (1)} &\ldots  &  X_{\cdot 1}^{D, D(1)} \end{matrix}  &
0
\end{pmatrix}\]
\[
H_{\cdot 1}^{(1)}  =  \begin{pmatrix}
0 &
\begin{matrix}
 X_{\cdot 1}^{1,1(1)} & \ldots & X_{\cdot 1}^{1, D(1)}  \\ \vdots  &   & \vdots\\  X_{\cdot 1}^{D,1 (1)} &\ldots  &  X_{\cdot 1}^{D, D(1)} \end{matrix}  \\
\begin{matrix}
 X_{1\cdot }^{1,1(1)} & \ldots &X_{1\cdot}^{D, 1(1)}  \\ \vdots  &   & \vdots\\  X_{1 \cdot }^{1, D (1)} &\ldots  &  X_{1 \cdot}^{D, D(1)} \end{matrix}  &
0
\end{pmatrix}\]
and $X_{1\cdot}^{c,d(1)}$ ($X_{\cdot1}^{c,d(1)}$) is the first row (column) of the $cd^{th}$ block of $X_N$ with the first entry set to zero.

Each entry of $\left(  H_{1\cdot}^{(1)} R_N^{(1)}  H_{\cdot 1}^{(1)}  \right)  $ is a sum of quadratic forms.

Since the set $\{X_{1\cdot}^{c,d(1)}\}_{1 \leq c,d \leq D}$ is independent from $\{X_{\cdot 1}^{c,d(1)}\}_{1 \leq c,d \leq D}$, \\
$\E[\left(  H_{1\cdot}^{(1)} R_N^{(1)}  H_{\cdot 1}^{(1)}  \right)^{cd}] = 0$, when exactly one of $c,d$ is less than or equal to $D$.

If $c,d$ are both less than or equal to $D$ or both greater than $D$ then $\left(  H_{1\cdot}^{(1)} R_N^{(1)}  H_{\cdot 1}^{(1)}  \right)  $ can have a non-zero expectation. We consider case when $1\leq c, d \leq D$, the other case is similar.

\begin{align*}
\left(  H_{1\cdot}^{(1)} R_N^{(1)}  H_{\cdot 1}^{(1)}  \right)^{cd} &= \sum_{1\leq e,f\leq D} X_{1\cdot}^{c, e(1)} R^{e+D,f+D(1)} X_{ 1\cdot}^{d, f(1)} \\
&=\sum_{1\leq e,f\leq D} \sum_{i,j>1 } X_{1 i}^{c, e} R_{ij}^{e+D,f+D(1)} X_{1 j}^{d, f}.
\end{align*}
Then averaging with respect to the first row and column of $X_N$ gives:
\[\E[ \sum_{1\leq e,f\leq D} \sum_{i,j>1 } X_{1 i}^{c, e} R_{ij}^{e+D,f+D(1)} X_{1 j}^{d,f}]
= \delta_{c,d}  \sum_{e=1}^D \E[ \Tr(R^{e+D,e+D(1)})]  \frac{ g_{ce}^2 }{N}.
\]
Recalling the definition of $\Sigma$, we see that $\E[  H_{1\cdot}^{(1)} R_N^{(1)}  H_{\cdot 1}^{(1)} ] = \Sigma(\E[ M^{(1)}_N] )$.
Taking the expectation of \eqref{Rschur} leads to the equation:
\[  \E[M_N] = \E[ R_{N;11}] = - \E[ \left(  q \otimes I_D
+ \Sigma( \E[M_N] )
+ Z_N
 \right)^{-1}]
\]
where

\[ Z_N =    \left(  H_{1\cdot}^{(1)} R_N^{(1)}  H_{\cdot 1}^{(1)}  \right)    - \Sigma(\E[ M_N] )  -H_{11}   .\]

\begin{lemma} 
\label{errorest}
The expectation of the norm of the error term $Z_N =    \left(  H_{1\cdot}^{(1)} R_N^{(1)}  H_{\cdot 1}^{(1)}  \right)    - \Sigma(\E[ M_N] )  -H_{11} $ is bounded by
$ \E[ \|Z_N \| ] = O(N^{-1/2} |\Im(\eta)|^{-1} )$.
\end{lemma}

\begin{proof}
We rewrite
\begin{align*}
Z_N = & \left(  H_{1\cdot}^{(1)} R_N^{(1)}  H_{\cdot 1}^{(1)}  \right) - \Sigma( M_N^{(1)} )
+ \Sigma( M_N^{(1)} ) -  \Sigma( M_N )
 + \Sigma( M_N ) -\Sigma(\E[ M_N] ) -H_{11}
\end{align*}
By the triangle inequality, it suffices to bound each term individually. Furthermore, since $Z_N$ is a matrix, to show its norm goes to zero it suffices to show each entry goes to zero individually. By assumption $\E[ \|H_{11} \| ] = O(N^{-1/2})$.

As noted earlier, $\left(  H_{1\cdot}^{(1)} R_N^{(1)}  H_{\cdot 1}^{(1)}  \right) - \Sigma( M_N^{(1)} ) $ is a sum of quadratic forms with zero mean. After applying Jensen's inequality, we directly compute the second moment of each quadratic form.

Letting $Y,Z$ be $X_{1*}^{c,d}$ or $X_{*1}^{c,d}$ for some $c,d$
\begin{align*}
\E[ \left |\left(  Y^* R_N^{a,b(1)}  Z \right) -  \delta_{Y,Z} \tr_N( R_N^{a,b(1)} ) \right| ] \leq C \left( \frac{1}{N } \tr_N(R^{a,b(1)*}R^{a,b(1)}) \right)^{1/2} \leq \frac{C|\Im(\eta)|^{-1} }{ N^{1/2} }.
\end{align*}

To bound $\Sigma( M_N^{(1)} - M_N )$, we use that the normalized partial trace of a matrix is bounded by its norm times its rank, giving the bound:
\begin{equation}
\label{submatrix}
|M_N^{cd}(q)-M_N^{cd(k)}(q)| \leq \frac{4D |\Im(\eta)|^{-1}}{N } .
\end{equation}
Applying $\Sigma$ only changes the value of the constant.



Finally, we use a martingale decomposition to estimate the final term. Let $\E_{j}$ denote averaging over the first $j$ rows and columns of $H_N$. We briefly use the notation $M_N^{cd[j]}$ to be as $M_N^{cd}$ but with the $j^{th}$ row and column of $H_N$ set equal to zero.
\begin{align*}
\E[ | M_N^{cd} - \E[ M_N^{cd}]| ] &\leq \E[ | M_N^{cd} - \E[ M_N^{cd}]|^2 ]^{1/2} \\
&= \E[| \sum_{j=1}^{N} (\E_{j-1} - \E_{j})[ M_N^{cd} -  M_N^{cd[j]}] |^2 ]^{1/2} \\
&= \left( \sum_{j=1}^{N} \E[ | (\E_{j-1} - \E_{j})[ M_N^{cd} -  M_N^{cd[j]}] |^2 ] \right)^{1/2} \\
&=O(N^{-1/2} |\Im(\eta)|^{-1}) 
\end{align*}
and we used an estimate similar to \eqref{submatrix} in the final line.

\end{proof}

Having bounded the error term, we now bound the difference between $\E[M_N(q)]$ and $M(q)$. Before we begin, recall that there is a unique solution to equation \eqref{Meq} with positive imaginary part.
\begin{lemma}
\label{expest}
Let $M_N(q)$ be as in \eqref{defMN} and $M(q)$ be as in \eqref{Meq}, then
\[ \| \E[M_N(q)] - M(q) \| = O(|\Im(\eta)|^{-5} N^{-1/2}) \]
\end{lemma}

\begin{proof}

We assume that $|\Im(\eta)|^{} > 4^{1/4} N^{-1/8} $ otherwise the trivial bound of $|\Im(\eta)|^{-1}$ on the norm of the resolvent gives the desired estimate.

Rewriting
\begin{align} 
\label{Minvert}
\E[M_N] =& \E[ -(q\otimes I_D + \Sigma(\E[M_N]) +Z_N )^{-1}] \\
=&  -(q\otimes I_D + \Sigma(\E[M_N])  )^{-1}  \nonumber \\
&+ (q\otimes I_D  + \Sigma(\E[M_N]) )^{-1} \E[ Z_N (q\otimes I_D + \Sigma(\E[M_N]) +Z_N  )^{-1} ] \nonumber
\end{align}
Let  \[Z'_N = (q\otimes I_D + \Sigma(\E[M_N]) )^{-1}Z_N(q\otimes I_D + \Sigma(\E[M_N]) +Z_N )^{-1} .\]

\[ \E[\| Z_N (q + \Sigma(\E[M_N]) +Z_N  )^{-1} \| ] = \E[\| Z_N R_{N;11}\| ]   \leq 1/2\]
 so $\E[1+ Z_N R_{N;11}]$ is invertible. After solving for $(q\otimes I_D + \Sigma(\E[M_N])  )^{-1}$ in \eqref{Minvert}, this estimate implies
\[  \| (q\otimes I_D + \Sigma(\E[M_N])  )^{-1} \| \leq 2 |\Im(\eta)|^{-1} \text{ and } \|  Z_N'  \| \leq C |\Im(\eta)|^{-2} \| Z_N \| .\]

Let $q_N := q\otimes I_D + \Sigma(Z_N')$, which has imaginary part bigger that $\Im(\eta)/2>0$ for large $N$. Then
\[ \E[M_N(q) ] - Z_N' = -(q_N + \Sigma(\E[ M_N( q)] - Z_N') )^{-1} \]
then evaluating \eqref{Meq} at $q_N$ and using the uniqueness of the solution shows:
\[ \E[M_N(q)]  - Z_N' = M(q_N) .\]
Finally:
\[ \|  \E[M_N(q)] - M(q)\| =  \| M(q_N) - M(q)  + Z_N' \| = O(|\Im(\eta)|^{-5} N^{-1/2})  \]

\end{proof}

We conclude with a concentration estimate.
\begin{lemma}
\label{concen} Let $M_N(q)$ be as in \eqref{defMN} then almost surely
\[ \|M_N(q) - \E[M_N(q)] \| \leq N^{-1/4} | \Im(\eta)|^{-1}   \]
\end{lemma}
\begin{proof}

We use McDiarmid's inequality, which states:\\
Suppose $x_1, x_2, \dots, x_n$ are independent random vectors and assume that
$f$ satisfies
\[\sup_{x_1,x_2,\dots,x_n, x_i'} |f(x_1,x_2,\dots,x_n) - f(x_1,x_2,\dots,x_{i-1},x_i', x_{i+1}, \dots, x_n)|
\le c_i \qquad \text{for} \quad 1 \le i \le n \]

then any $\varepsilon > 0$:
\[\Pr ( |E[f(x_1, x_2, \dots, x_n)] - f(x_1, x_2, \dots, x_n)| \ge \varepsilon )
\le 2 \exp \left( - \frac{2 \varepsilon^2}{\sum_{i=1}^n c_i^2} \right). \;
\]
%
%

We apply the inequality choosing $x_i$ to be the $i^{th}$ column of $X_N$. Once again, it suffices to prove the inequality for each entry of $M_N$, let $1\leq c,d \leq D$. Let $\hat M_N^{cd}(q)$ be the same as $M_N^{cd}(q)$ except with the $i^{th}$ column of $X_N$ resampled, then using that the difference between $X_N$ and its resampled version is rank two, gives the bound:
\[ | M_N^{cd}(q) - \hat  M_N^{cd}(q) | \leq \frac{2 }{\eta N} \]
and McDiarmid's inequality implies
\[ \P( |M_N^{cd}(q) - \E[M_N^{cd}(q)] | > \varepsilon  ) \leq 2 e^{- 2^{-1} N |\Im(\eta)|^{2} \epsilon^2 } \]
or
\[ \P( |M_N^{cd}(q) - \E[M_N^{cd}(q)] | > N^{-1/4} | \Im(\eta)|^{-1} ) \leq C e^{- N^{1/2} } \]
the Borel-Cantelli lemma completes the proof.

\end{proof}

\section{Logarithmic Integrability}
\label{logi}
By Markov's inequality, to prove \eqref{logieq} it suffices to show almost surely
\begin{equation} \label{logip}  \limsup_{N\geq 1}  \int_0^{\infty} x^{p} d \nu_{z,N}(x) < \infty \text{ and } \limsup_{N\geq 1}  \int_0^{\infty} x^{-p} d \nu_{z,N}(x) < \infty \end{equation}
for some $p>0$.
Let $s_{N,z} \leq s_{N-1,z} \leq \ldots \leq s_{1,z}$ be the ordered singular values of $X_N-z$. For  $0<p\leq 2$, the first inequality follows from
\[ \int_0^{\infty} x^{2} d \nu_{z,N}(x) = \frac{1}{N} \sum_{i=1}^{N} s_{i,z}^2 \leq |z|^2 +  \frac{1}{N} \sum_{i=1}^{N} s_{i,0}^2 = |z|^2 + \frac{1}{N} \Tr(X_N X_N^*) \]
Which is almost surely finite by the strong law of large numbers (applied separately to each block).

The second inequality will follow from a bound on the least singular value proved in Corollary \ref{lsv2} and control of moderate singular values proved in Lemma \ref{wegner}.

The bound on the least singular value only requires the entries of the random matrix, $X_N$, to have finite variance. To prove Corollary \ref{lsv2} we use the following theorem:

\begin{theorem}[\cite{TV}, Theorem 2.1]
\label{TVt}
Let $a, c_1$ be positive constants, and let $x$ be a complex-valued random variable with non-zero finite variance. Then there are positive constants $b$ and $c_2$ such that the following holds: if $Y_n$ is the random matrix of order $n$ whose entries are iid copies of $x$, and $M$ is a deterministic matrix of order $n$ with spectral norm at most $n^{c_1}$, then,
\[ \P( \sigma_n (M+Y_n) \leq n^{-b}) \leq c_2 n^{-a}.\]
\end{theorem}

\begin{corollary}
\label{lsv2}
Let $a, c_1'$ be positive constants.
Let $D$ be an $n \times n$  random matrix with iid entries and finite variance and $M$ be an $n \times n$ deterministic matrix of spectral norm at most $N^{c_1'}$. Let $A$ be an $(N-n) \times (N-n)$ matrix, $B$ be an $(N-n)\times n$ matrix, and $E$ be an $n\times (N-n)$ matrix. Assume $N/n$ is equal to a non-zero constant plus a term that is $o(N)$.

Assume there exist $b''$ and an event occurring with probability $1-N^{-a}$ on which $A^{-1}$  has norm bounded by $N^{b''}$ and $\|B\|, \|E\|$ have norm $N^{b'''}$. Furthermore assume the $D$ is independent of $A$, $B$ and $E$.

Let $Y = \begin{pmatrix} A & B \\ E & D \end{pmatrix}$ and let $\mathcal{M}= \begin{pmatrix} 0 & 0 \\ 0 & M \end{pmatrix}$,
then there are positive constants $b'$ and $c_2'$ (depending on $a$, $c_1'$, $b''$ and $b'''$) such that
\[ \P( \sigma_N( \mathcal{M}+Y) \leq N^{-b'}) \leq c_2' N^{-a}.\]
\end{corollary}
\begin{proof}
Recall that the least singular value of a matrix $X$ can be characterized as $\min_{\|u\|=1} \|X u\|$ or $\|X^{-1}\|^{-1}$.
Let $u$ be a unit vector such that $\|(\mathcal{M}+Y)u\| = \sigma_N(\mathcal{M}+Y)$ and partition $u$ as $\begin{pmatrix} u_1 \\ u_2 \end{pmatrix} $, where $u_1$ is $N-n$ dimensional and $u_2$ is $n$ dimensional. Let $ v:= (\mathcal{M}+Y) u $ and partition $v$ similarly as $\begin{pmatrix} v_1 \\ v_2 \end{pmatrix} $. Expressing $v = ( \mathcal{M}+Y)u$ in terms of the blocks leads to \[ A u_1 + B u_2 = v_1 \text{ and } E u_1 + (D+M) u_2 = v_2. \]

If $\|u_2\| = o(N^{-b''-b'''})$ then $A u_1 = v_1 + o(N^{-b''})$ and our assumption on the least singular value of $A$ implies, $\| v_1 \| \geq c_3 N^{-b''}$ with probability $1 - N^{-a}$ for some $c_3>0$.

Otherwise, using the assumption that there is an event of probability $1-N^{-a}$ on which $A$ is invertible, we solve for $u_1$. Substituting into the second equation gives
\[ ( M + D - E A^{-1} B) u_2 = v_2 - EA^{-1} v_1. \]
Since $E A^{-1} B$ is independent of $D$ and has norm bounded by $N^{b''+2b'''}$ with probability $1-N^{-a}$, Theorem \ref{TVt} implies there exist a $b$ such that $\| v_2 - E A^{-1} v_1 \| \geq n^{-b} \| u_2 \|$. Implying that $\| v \| \geq c_4  N^{-b-2b''-2b'''}$ for some $c_4>0$.

\end{proof}

To apply this corollary to our case, we must show the operator norm of a rectangular matrix with iid random entries grows polynomially. Using the assumption that the entries have finite variance along with Markov's inequality and the union bound implies there is an event occurring with probability $1-N^{-a}$ such that all the entries of a block of $X_N$ are bounded by $C N^{a/2}$, for some $C$. Then bounding the Frobenius norm shows that on this event the operator norm the norm of a block is $O( N^{a/2+1})$. 
We remark that this estimate, although sufficient, is far from optimal. For instance, if the entries of a random matrix are bounded then with probability $1 - O(N^{-a})$ the operator norm of a matrix with iid entries is $O(1)$, see for instance \cite[Section 2.3]{T}. 

Concatenating matrices only increases the operator norm by a constant factor. Thus the corollary can inductively be applied to the matrices formed by the $k$ upper left blocks for $1\leq k\leq D$ of $X_N$ giving the almost sure polynomial bound on the least singular value.


The following lemma gives control on moderate singular values.

\begin{lemma}
\label{wegner}
There exist a $\gamma> 0$ and $C>0$ such that $\nu_{z,N}([0,I]) \leq C I$ for every $I \geq N^{-\gamma}$.
\end{lemma}

\begin{proof}

We first observe that for $t>0$, $1_{\{|x| \leq t\}}(x) \leq 2 t \Im(1/(x-\sqrt{-1}t)$. Furthermore $a(q)$ is bounded (see proof of Theorem \ref{supp}) in the upper half plane,
therefore choosing $q$ such that $\eta = \sqrt{-1} t$ and applying Proposition \ref{mainest} gives
\[ \nu_{z,N}([0,t]) \leq 2 t \Im( a_N(q) ) \leq  2 |t| ( (t^{-5} N^{-1/4}) +  |a(q)| )   \]
this term is bounded by $C t$ for all $t> N^{-1/20}$.

\end{proof}

To complete the proof of \eqref{logip} let $b>0$, be such such that the least singular value of $X_N$ is almost surely greater than $N^{-b}$ and let $\gamma$ be such that $\nu_{z,N}([0,I]) \leq C I$ for every $I \geq N^{-\gamma}$.

\begin{align*} \int_0^{\infty} x^{-p} d \nu_{z,N}
&= \frac{1}{N} \sum_{i=1}^N s_{i,z}^{-p}
\leq \frac{1}{N} \sum_{i=1}^{N-N^{-\gamma}} s^{-p}_{i,z} + \frac{1}{N} \sum_{i=N-N^{-\gamma}}^N s^{-p}_{i,z}  \\
& \leq C \frac{1}{N}  \sum_{i=1}^{N} (\frac{i}{N})^{-p} + N^{-\gamma} N^{b p}
\end{align*}
The first term uses the moderate singular value bounded from Lemma \ref{wegner}. This Riemann sum is finite for $0 < p < 1$. The second term is bounded using from the least singular value bound and is $o(1)$ for $p$ sufficiently small.

\section{Analysis of the limiting measure}
\label{analysis}

Let $\lambda$ be the Perron-Frobenius eigenvalue of $G$, recall this is also the spectral radius of $G$.

\begin{theorem}
\label{supp}
The support of $\mu$ is a disk with radius $\sqrt{\lambda}$.
\end{theorem}

Before starting the proof, we begin by recalling \eqref{aeq} in the homogeneous case. 
Let $a_{iid}$ be the solution to  \eqref{aeq} in the case that there is just one block of size $N$ with variance $\sigma^2$, and $h_{iid} =  a_{iid}/\sqrt{-1}$  then
\[ h_{iid} = \frac{  \sigma^2 h_{iid} + t}{ |z|^2 + ( \sigma^2 h_{iid} + t)^2 } \]
which has solution as $t \to 0$
\[ \lim_{t \to 0} h_{iid} = \begin{cases} \frac{1}{\sigma^2} \sqrt{\sigma^2 - |z|^2 } & \text{ if } |z|^2 \leq \sigma^2 \\ 0 & \text{ otherwise}  \end{cases} \]

In our proof, we manipulate the equations defining $a$ to resemble the iid case, from which the spectral radius can be extracted.

\begin{proof}

The equations defining $a$ and $\widehat a$, \eqref{aeq} and \eqref{waeq}, only depend on $|z|$. As noted in Section 3, $\mu$ can be recovered $a$ and hence $\mu$ is a radially symmetric measure.

Fix $|z|^2 \leq \lambda$ we will now show there exist a $C_z > 0$ such that $\lim_{t \to 0} |a(\sqrt{-1}t,z)| > C_z$. In what follows many of the variables depend on $\eta$ and $z$ but this dependence will often be suppressed.

Let $h_c := a_c/ \sqrt{-1} $. Since $a_c$ is purely imaginary with non-negative imaginary part, $h_c \geq 0$.
Let $h$ be the vector with $c^{th}$ entry $h_c$ and define $\widehat{h}$ similarly. Multiplying \eqref{aeq} by $\sqrt{-1}$ leads to the equations:
\begin{equation}
\label{heq}
 h_c = \frac{ [G h]_c + t}{|z|^2 + ( [\widehat  G \widehat h]_c + t) ([Gh]_c + t) } ,
\quad
 \widehat h_c = \frac{[\widehat G \widehat h]_c + t}{|z|^2 + ( [\widehat G \widehat h]_c + t) ([Gh]_c + t)  }
 \end{equation}
It will be more convenient to work with these equations, as all the quantities are real.

Note that for each $c$, $h_c$ and $\widehat h_c$ are bounded uniformly for $t >0$. Indeed, if there was an index $c$ such that $h_c \to \infty$ then the right side of the first equation in \eqref{heq} is also unbounded, implying that $[\widehat G \widehat h]_c \to 0$. Since the entries of $\widehat G$ are positive this also implies $ \widehat h_d \to 0$ for each $d$, contradicting the fact that $ \sum_{c=1}^D \alpha_c h_{c}(\sqrt{-1}t,z) =\sum_{c=1}^D \alpha_c \widehat h_{c}(\sqrt{-1}t,z) $. 

Since $|a(\sqrt{-1}t,z)| =  \sum_{c=1}^D \alpha_c h_{c}(\sqrt{-1}t,z) =\sum_{c=1}^D \alpha_c \widehat h_{c}(\sqrt{-1}t,z) $ it suffices to show for each $t$ small, at least one of the $h_c$ is bounded from below.

Let $v$ be a Perron-Froebinius eigenvector of $G$, we shall shortly see the choice of  normalization is not important. For each $c$, let $v^c = \frac{h_c}{v_c} v$ be the Perron-Froebinius eigenvector of $G$, normalized such that $v_c^c = h_c$. 

Then we rewrite $ [Gh]_c$ as:

\begin{align*}
[Gh]_c &= \sum_{d} G_{cd}( h_d + v^{c}_d - v^{c}_d) \\
&=\lambda v^c_c + \sum_{d \not=c} G_{cd} (h_{d} - v_d^c)  \\
&=  \lambda h_c + \sum_{d \not=c} G_{cd} (v_{d}^d - v_d^c)  
\end{align*}


For each $t$, let $c$ be an index such that $\frac{h_c}{v_c}$ is the smallest. Implying that $v_d^c \leq v_d^d$ for all $d$ or $(v_{d}^d - v_d^c)\geq 0$ for all $d$.

Using this $c$, let $t_1 = t + \sum_{d \not=c} G_{cd} (v_{d}^d - v_d^c)$.

\begin{equation}
\label{wlambda}
 h_c = \frac{ \lambda h_c + t_1 }{ |z|^2  + ([\widehat  G \widehat h]_c + t) (\lambda h_c + t_1 )  }
\end{equation}

or
\[ ([\widehat  G \widehat h]_c + t)(\lambda h_c + t_1 ) =   (\lambda - |z|^2) + t_1 h_c^{-1} .\]

Let $ K = \max\{\max_{d} h_d,\sup_{d,0<t<1} ([\widehat  G \widehat h]_d + t)$\}.

To show that $h_c$ is bounded from below we consider two cases. First, assume that $t_1 \leq \frac{\lambda - |z|^2}{2 K}$. Rearranging the above equation gives
\[  ([\widehat  G \widehat h]_c + t) \lambda h_c \geq  (\lambda - |z|^2) - ([\widehat  G \widehat h]_c + t) t_1 \]
or 
\[ h_c \geq \frac{\lambda - |z|^2}{2 K \lambda}. \]

Alternatively, assume $t_1 > \frac{\lambda - |z|^2}{2 K}$ then 
\[  h_c^{-1}  = \left| \frac{  ([\widehat  G \widehat h]_c + t)(\lambda h_c + t_1 ) - (\lambda -|z|^2 ) }{t_1} \right| \leq \frac{2K^3 \lambda}{\lambda-|z|^2}+ K+ 2 K  . \] 
We conclude that $|a(\sqrt{-1}t,z)|$ is uniformly bounded away from $0$ for all positive $t$.\\




To conclude the proof of the lemma, we now fix $|z|^2 > \lambda$ and consider the vector $(h, \widehat h)$ formed by concatenating the vectors $h$ and $\widehat h$. We will show for any $1>\epsilon>0$ there exist a $C$ such that $\|(h,\widehat h)\| < \epsilon  \|(h,\widehat h)\| + Ct  $. Implying $h_c, \widehat h_c \to 0 $ as $t \to 0$ for each $1\leq c\leq D$. 

Let $T$ be the function from $\R^{2D} \to \R^{2D}$ defined by
\begin{equation*}
\label{Th}
  [T(h,\widehat h) ]_c=  \frac{ \frac{1}{|z|^2} ( [G h]_c + t) }{1 + ( \frac{1}{|z|^2} ( [\widehat G \widehat h]_c + t) )( [Gh]_c + t) )      }
  \end{equation*}
for $1 \leq c \leq D$ and
\[  [T(h,\widehat h) ]_c =  \frac{ \frac{1}{|z|^2} ( [\widehat G  \widehat h]_c + t) }{1 +  ( \frac{1}{|z|^2} ( [\widehat G \widehat h]_c + t) )( [Gh]_c + t) )      } \]
for $d+1 \leq c \leq 2D$.

If $(h,\widehat h)$ is a solution to \eqref{heq} then $T((h,\widehat h))=(h,\widehat h)$ and therefore $T^k((h,\widehat h)) = (h,\widehat h)$.
Since $|z|^2 > \rho(G)$, for any $\epsilon>0$ there exist a $k$ such that $\| G^k \|/|z|^{2k} \leq \epsilon$.

Then using that the denominator is always greater than 1, we bound $T^k(h,\widehat h)$ in terms of $T^{k-1}(h,\widehat h)$ by
\[ [T^{k}(h,\widehat h)]_c \leq \frac{1}{|z|^2}  [G  T^{k-1}(h,\widehat h) ]_c + t   \]

Iterating this estimate leads to
\[ [T^{k}(h,\widehat h)]_c \leq   \left(\frac{1}{|z|^2} \right)^k [G^k h]_c +  \left[\frac{1}{|z|^2}(I + \frac{G}{|z|^2} + \ldots \frac{G^k}{|z|^{2k}} ) \tilde{t}  \right]_c \leq \epsilon \|(h,\widehat h)\| + Ct    \]
where $ \tilde{t}$ is vector with all entries equal to $t$. 
\end{proof}

\subsection*{Acknowledgments}
The authors would like to thank Nick Cook for useful discussions.  The authors would also like to thank the anonymous referee for valuable comments and careful reading of the manuscript.

\bibliographystyle{abbrv}


\end{document}